\documentclass[twoside,11pt,reqno]{amsart}
\usepackage{tikz}

\usepackage[bookmarks,
bookmarksnumbered,%
 colorlinks=true,%
 linkcolor=red,%
 citecolor=blue,%
 filecolor=blue,%
 urlcolor=blue,%
]
{hyperref}

\usepackage{amssymb, amsmath, latexsym, mathtools}

\makeatletter
\@namedef{subjclassname@2020}{\textup{2020} Mathematics Subject Classification}
\makeatother

\DeclareMathAlphabet{\curly}{OT1}{rsfs}{n}{it}
 
\DeclareMathOperator{\Dim}{dim}

\DeclareMathOperator{\rank}{rank}

\DeclareMathOperator{\im}{Im}

\DeclareMathOperator{\Fix}{Fix}
\DeclareMathOperator{\tr}{tr}

\DeclareMathOperator{\inc}{in}
\DeclareMathOperator{\pr}{pr}
\DeclareMathOperator{\coker}{Coker}
\DeclareMathOperator{\Sm}{Sm}

\DeclareMathOperator{\Sq}{Sq}
\DeclareMathOperator{\Bl}{Bl}
\DeclareMathOperator{\Tors}{Tors}
\DeclareMathOperator{\interior}{Int}

\DeclareMathOperator{\defi}{\mathfrak D}
\DeclareMathOperator{\Conj}{c}

\DeclareMathOperator{\Ker}{Ker}

\def\dim{\mbox{dim}}
\def\ra{\rightarrow}

\def\CC{\mathbb{C}}
\def\PP{\mathbb{P}}
\def\QQ{\mathbb{Q}}
\def\ZZ{\mathbb{Z}}
  
\def\RR{\mathbb{R}}

\def\FF{\mathbb{F}}

\def\s-{\setminus}

\def\HH{\mathbb{H}}

\newtheorem{thm}{Theorem}[section]
\newtheorem{prop}[thm]{Proposition}
\newtheorem{lemma}[thm]{Lemma}
\newtheorem{defn}[thm]{Definition}

\newtheorem{cor}[thm]{Corollary}
\newtheorem{lem}[thm]{Lemma}

\newtheorem{rmk}[thm]{Remark}

\numberwithin{equation}{section}

\begin{document}

\title[Smith--Thom deficiency bounds]
{Lower bounds for the Smith--Thom deficiency of Hilbert squares}

\author[Kharlamov]{Viatcheslav Kharlamov}

\address{
        IRMA UMR 7501, Strasbourg University, 7 rue Ren\'e-Descartes, 
        67084 Strasbourg Cedex,  FRANCE}

\email{kharlam@math.unistra.fr}

\author[R\u asdeaconu]{Rare\c s R\u asdeaconu}

\address{        
        Department of Mathematics, 1326 Stevenson Center, Vanderbilt University, 
        Nashville, TN 37240, USA\newline
        \indent ``Simion Stoilow'' Institute of Mathematics of the Romanian Academy, 
        21 Calea Grivi\c tei,  010702 Bucharest, Romania}

\email{rares.rasdeaconu@vanderbilt.edu}



\begin{abstract}
We establish lower bounds for the Smith--Thom deficiency of the Hilbert square 
of a maximal nonsingular real projective variety. These bounds recover the 
known surface case and provide new obstructions to Smith--Thom maximality 
in every dimension at least two.  As applications, we prove that the Hilbert square 
of any real abelian variety of dimension at least two is not Smith--Thom maximal. 
We further show that the deficiencies of the Hilbert squares of maximal real abelian 
varieties grow exponentially with the dimension. We also obtain nonmaximality 
results for the Hilbert squares of certain Cartesian products.
\end{abstract}

\dedicatory{In memoriam V.~I.~Arnold}

\maketitle

\thispagestyle{empty}

\setcounter{tocdepth}{2}


\section{Introduction}
\label{intro}


The Smith--Thom inequality is one of the basic topological constraints on a real 
algebraic variety $X:$
$$
 \dim \, H_*(X(\RR);\FF_2)\leq \dim\, H_*(X(\CC);\FF_2).
$$
The difference
$$
 \defi(X)
 =\dim \, H_*(X(\CC);\FF_2)-\dim \, H_*(X(\RR);\FF_2)
$$
is called the \emph{Smith--Thom deficiency} of $X.$  The extremal case $\defi(X)=0$ is 
of particular interest: one then says that \(X\) is \emph{Smith--Thom maximal}, or simply 
\emph{maximal}.

The behavior of Smith--Thom maximality under natural geometric constructions is 
subtle. This is already apparent for the Hilbert square $X^{[2]},$ the canonical
resolution of the symmetric square $X^{(2)}.$ When $X$ is maximal, its symmetric 
square is always maximal \cite{franz}, while its Hilbert square need not be maximal. 
This phenomenon was detected in our previous work \cite{loss-surfaces} in the 
case of surfaces. For surfaces with vanishing odd cohomology, the results of 
\cite{loss-surfaces} have since been extended from Hilbert squares to higher 
Hilbert schemes of points in \cite{bfgk}. In higher dimensions, we obtained in 
\cite{hypersurfaces} a necessary and sufficient condition for the Smith--Thom 
maximality of the Hilbert square of nonsingular projective complete intersections 
and showed that, with a small number of exceptions, a nonsingular real projective 
complete intersection of even dimension cannot have a maximal
Hilbert square.

The purpose of the present paper is to generalize the obstruction to Smith--Thom 
maximality obtained in \cite{loss-surfaces} and \cite{hypersurfaces} and to establish 
explicit lower bounds for $\defi(X^{[2]})$ in every dimension at least $2$ in terms 
of the $\FF_2$-Betti numbers of $X(\CC)$ and $X(\RR)$. Our main estimate is the 
following:

\begin{thm}
\label{defect-mv}
Let $X$ be a maximal nonsingular real projective variety of dimension
$n\geq 2$ such that $\Tors_2 H_*(X(\CC);\ZZ)=0.$  Then
\begin{equation}
\label{defect-formula}
\defi(X^{[2]})\geq
\begin{cases}
\displaystyle
4\left(
\sum_{k=0}^{\frac{n-1}{2}}\sum_{i=0}^{k}\beta_i(X(\RR))
-\sum_{l=0}^{\frac{n-1}{2}}\sum_{j=0}^{2l}\beta_j (X(\CC))-\frac12\beta_{\rm odd}(X(\CC))
\right),\,
 \text{if $n$ is odd},\\
\displaystyle
4\left(
\sum_{k=1}^{\frac n2}\sum_{i=0}^{k-1}\beta_i(X(\RR))
-\sum_{l=1}^{\frac n2}\sum_{j=0}^{2l-1}\beta_j(X(\CC))
\right),\,
 \text{if $n$ is even}.
\end{cases}
\end{equation}
\end{thm}

For surfaces, Theorem~\ref{defect-mv} recovers \cite[Theorem~1.5(2)]{loss-surfaces}.  
In higher dimensions, the positivity of the right-hand side of \eqref{defect-formula} 
gives an explicit obstruction to the Smith--Thom maximality of $X^{[2]}.$ Unlike the 
complete-intersection criteria obtained previously in \cite{hypersurfaces}, the present 
estimates apply to arbitrary maximal nonsingular real projective varieties with no $2$-torsion 
in integral homology and give quantitative lower bounds for the deficiency.

The proof of Theorem \ref{defect-mv} combines the geometric decomposition of the 
real locus of the Hilbert square introduced in \cite{hypersurfaces}, the exact formulas for 
the Smith--Thom deficiency in Theorems \ref{even-rk-thm} and \ref{odd-rk-thm}, and 
estimates for the ranks of the corresponding Mayer--Vietoris maps.

The first application concerns real abelian varieties.  For Smith--Thom maximal abelian 
varieties, the general lower bound given by Theorem \ref{defect-mv} is already positive 
in every dimension except $2$ and $4$.  A refined analysis supplies the missing contribution 
in these exceptional dimensions and gives a uniform asymptotic estimate.

\begin{thm}
\label{tori}
If $(X,\Conj)$ is a real abelian variety of dimension $n\geq 2,$ then its Hilbert square $X^{[2]},$ 
equipped with the induced real structure, is not Smith--Thom maximal. Furthermore, for 
any sequence $(X_n,\Conj_n)$ of maximal abelian varieties with $\Dim_\CC X_n=n$,
$$
 \liminf_{n\to\infty}
 \frac{\ln\defi(X_n^{[2]})}{n}\geq 2\ln 2.
$$
\end{thm}

The estimate \eqref{defect-formula} also detects nonmaximality for Cartesian products.  
We record two representative consequences.

\begin{thm}
\label{product-not-max}
Let $C$ be a nonsingular real projective curve, and let $X$ and $Y$ be
nonsingular real projective surfaces.  Assume that
$\beta_1(X)=\beta_1(Y)=0$ and that $X(\RR)$ is disconnected.  Then 
\begin{itemize}
\item[1)] \((X\times C)^{[2]}\) is not Smith--Thom maximal;
\item[2)] \((X\times Y)^{[2]}\) is not Smith--Thom maximal.
\end{itemize}
\end{thm}

The paper is organized as follows.  Section~\ref{preliminaries} recalls the
Smith sequences and the elementary identities used later.  In
Section~\ref{cut-paste}, we describe the decomposition of
\(X^{[2]}(\RR)\), compute the Betti numbers of its pieces, and determine the
ranks of the relevant boundary maps.  Section~\ref{deficiency-mv-section}
contains the exact Mayer--Vietoris formulas and the proof of
Theorem~\ref{defect-mv}.  Section~\ref{examples} treats real abelian varieties 
and Cartesian products.

\begin{rmk}
{\rm 
The preceding results and their proofs extend verbatim from the real algebraic 
setting to compact complex manifolds equipped with an antiholomorphic involution, 
with the Hilbert square replaced by the corresponding Douady space.
}
\end{rmk}


\subsection*{Acknowledgments} 
This work was begun at the Max Planck Institute for Mathematics in Bonn and
completed at the University of Strasbourg.  We thank both institutions for
their hospitality and support. The second author also acknowledges the support 
of Vanderbilt University during the preparation of this work.


\subsection*{Notation and conventions}


\begin{itemize}

\item[1)] By a complex variety equipped with a real structure, we mean a pair
$(Y,\Conj),$ where $Y$ is a complex variety and
$\Conj\colon Y\ra Y$ is an antiholomorphic involution. When the
antiholomorphic involution is understood from the context, we will simply say that
$Y$ is defined over the reals.

\item[2)] Let \(Y\) be an algebraic variety defined over $\RR,$ and let
\(G=\operatorname{Gal}(\CC/\RR)\). The group $G$ is cyclic of order \(2\)
and acts on the set of complex points \(Y(\CC)\). Its nontrivial element acts
as an antiholomorphic involution, which we denote by $\Conj,$ and its fixed-point
set is precisely \(Y(\RR)\). Thus, \((Y(\CC),\Conj)\) is a complex variety
equipped with a real structure. To reconcile the notation traditionally used
for varieties equipped with real structures with that used for algebraic
varieties defined over \(\RR\), we will henceforth write \(Y\) for \(Y(\CC)\)
and retain \(Y(\RR)\) for its set of real points.

\item[3)] Unless explicitly stated otherwise, all homology and cohomology
groups are taken with coefficients in the field
\(\FF_2=\ZZ/2\ZZ\). We use \(\beta_i(\,\cdot\,)\) and
\(b_i(\,\cdot\,)\) to denote the Betti numbers with coefficients in
\(\FF_2\) and \(\QQ\), respectively. For convenience, we allow \(i\) to be
an arbitrary integer by setting \(\beta_i(\,\cdot\,)=0\) for \(i<0\).
We use \(\beta_*(\,\cdot\,)\) and \(b_*(\,\cdot\,)\) for the corresponding
total Betti numbers, while
\(\beta_{\rm odd}(\,\cdot\,)\) and
\(\beta_{\rm even}(\,\cdot\,)\) denote, respectively, the sums
$\displaystyle
\sum_{i\geq 0}\beta_{2i+1}(\,\cdot\,)
$
and 
$
\displaystyle\sum_{i\geq 0}\beta_{2i}(\,\cdot\,).
$

\end{itemize}

\bigskip


\section{Smith theory and auxiliary Betti-number identities}
\label{preliminaries}


\subsection{Smith theory and maximality}
\label{Smith.theory}


Most results cited in this section are due to P. A. Smith. Proofs can be found, e.g., in 
\cite[Chapter 3]{bredon} and \cite[Chapter 1]{dik}.

\smallskip

Throughout this section, we consider a topological space $X$ with a cellular involution 
$c:X\to X;$ that is, $X$ is a CW complex, $c$ maps cells to cells, and $c$ acts 
identically on each invariant cell.\footnote{This rather traditional condition that $X$ is 
a CW complex (or a simplicial complex, as in \cite{bredon}) can be relaxed at the cost 
of using \v Cech cohomology and assuming $X$ is a finite-dimensional locally compact 
Hausdorff space. In this paper, Smith theory is applied to smooth manifolds 
and smooth involutions, so the CW complex assumption is sufficient.} Let Let  
$F=\Fix c,\,\bar X=X/c,$ and let $\inc: F\hookrightarrow X$ and $\pr: X\ra \bar X$ 
denote the natural inclusion and projection, respectively.

\smallskip

Consider the following \emph{Smith chain complexes}:
\begin{align*}
\Sm_*(X)&=\Ker[(1+c_*)\, :S_*(X)\to S_*(X)],\\
\Sm_*(X,F)&=\Ker[(1+c_*)\, :S_*(X,F)\to S_*(X,F)].
\end{align*}
There exists a canonical identification 
$$
\Sm_*(X,F)=\im[(1+c_*)\, :S_*(X)\to S_*(X)].
$$ 
The homology groups $H_r(\Sm_*(X))$ and $H_r(\Sm_*(X,F))$ are called the 
\emph{Smith homology groups}.
The \emph{Smith sequences} are the long exact homology and cohomology 
sequences associated with the short exact sequence of complexes
\begin{equation}
\label{smith-sequence}
0\ra\Sm_*(X)\xrightarrow{\text{inclusion}}S_*(X)
\xrightarrow{1+c_*}\Sm_*(X,F)\ra 0.
\end{equation}

Moreover, $\Sm_*(X)$ canonically splits as $\Sm_*(X)=S_*(F)\oplus\im(1+c_*),$ and 
the transfer homomorphism $\tr^*:S_*(\bar X,F)\to\Sm_*(X,F)$ is an isomorphism 
\cite[Chapter 3]{bredon} (see also {\it op. cit.} for the cohomology version). In view of 
these identifications, the long exact sequences associated with  
\eqref{smith-sequence} yield:

\begin{thm}
\label{Smith.seq}
There are two exact sequences, called \emph{the homology and cohomology 
Smith sequences of $(X,c)$}, that are natural with respect to equivariant maps:
$$
\begin{gathered}
\cdots \ra H_{k+1}(\bar X,F)\xrightarrow[]{\Delta} H_k(\bar X,F)\oplus H_k(F)
  \xrightarrow{\tr^k+\inc_k} H_k(X)
 \xrightarrow{\pr_k} H_k(\bar X,F)\ra \rlap{\,},\\
\ra H^k(\bar X,F)\xrightarrow{\pr^k}H^k(X)\xrightarrow{{\tr_k}\oplus{\inc^k}} 
  H^k(\bar X,F)\oplus H^k(F)\xrightarrow{\Delta} H^{k+1}(\bar X,F)\ra\cdots  \rlap{\,.}
\end{gathered}
$$

The homology and cohomology connecting homomorphisms $\Delta$ are given by
$$
x\mapsto x\cap\omega\oplus\partial x\quad\text{and}\quad x\oplus f\mapsto
x\cup\omega+\delta f,
$$
respectively, where $\omega \in H^1(\bar X\setminus F)$ is the characteristic class of 
the double covering $X\setminus F\to\bar X\setminus F$. The images of ${\tr^*}+\inc_*$ 
and~$\pr^*$ consist of invariant classes:
$\im(\tr^*+\inc_*)\subset\Ker(1+c_*)$ and $\im(\pr^*)\subset\Ker(1+c^*)$.
\end{thm}

The following immediate consequences of  Theorem \ref{Smith.seq}, which we 
state in the homology setting, have obvious counterparts in cohomology.

\begin{cor}
\label{cor-smith}
Let $(X, c)$ be a topological space equipped with a cellular involution. Then
\begin{equation}
\label{Smith-dims}
\Dim~H_*(F)+2\sum_{k}\Dim\coker({\tr^k}+\inc_k)=\Dim~H_*(X).
\end {equation}
As a consequence, we have 
\begin{equation}
\label{Smith-ineq}
\Dim~H_*(F)\le\Dim~H_*(X)\quad{\text{\rm (Smith inequality)}}.
\end{equation}
\end{cor}

\begin{defn}
Let $(X,\Conj)$ be a topological space equipped with a cellular involution. The integer 
$$
\defi(X,\Conj)=2\sum_{p}\Dim\coker({\tr^p}+\inc_p)
$$
is called the Smith--Thom deficiency of  $(X,\Conj).$
If $\defi(X,\Conj)=0$, the topological space 
$X$ is called \emph{Smith--Thom maximal}, or simply a \emph{maximal space}, and  
$\Conj$ is called an \emph{maximal involution}.
\end{defn}
When the involution is understood from the context, it will be omitted from the notation 
of the Smith--Thom deficiency.

\smallskip 

Notice from Corollary \ref{cor-smith} that $X$ is maximal if and only if 
$\Dim~H_*(F)=\Dim~H_*(X),$ and from Theorem \ref{Smith.seq} we find the 
following characterization of maximality.

\begin{cor}
\label{maxSmith}
Let $(X,c)$ be a topological space equipped with a cellular involution. Then 
$X$ is maximal if and only if the sequence 
$$
0\ra H_{k+1}(\bar X, F)\xrightarrow{\Delta} H_{k}(\bar X, F)\oplus H_k(F)\ra H_k(X)\ra 0
$$
is exact
for every $k\geq 0$.
\qed
\end{cor} 

\begin{lem} \cite[Lemma 2.5]{hypersurfaces}
\label{RelativeQuotient}
If a $d$-dimensional space $(X, c)$ is maximal and $r\leq d,$ then
\begin{equation}
\label{relative-aux}
\beta_r(\bar X, F)=\sum_{k=r}^d (\beta_k(X)- \beta_k(F)).
\end{equation}
If, in addition, $d=2n$, $X$ is a smooth closed manifold, $c$ is a smooth involution, 
and each component of $F$ is $n$-dimensional, then we have
\begin{equation}
\label{relative-individual}
\beta_{r}(\bar X, F)=\sum_{k=r}^{2n}\beta_k(X), \quad\text{for every $r\ge n+1$,}
\end{equation}
and 
\begin{equation}
\label{relative-main}
\beta_*(\bar X, F)=\frac{n}2\beta_*(X).
\end{equation}
\qed
\end{lem}


\subsection{Auxiliary Betti-number identities}


We next recall the following elementary computations, which are used in 
the proof of Theorem \ref{defect-mv}.

\begin{lemma}  \cite[Lemma 2.7]{hypersurfaces}
\label{elementary-odd}
Let $X$ be a maximal real nonsingular projective variety of odd dimension $n$. 
Then the following relations hold:
\begin{align}
\sum_{l=0}^{\frac{n-1}{2}}\sum_{i+j=2l}\beta_i(X(\RR))\beta_j(X(\RR))
=&\frac14 \beta^2_*, \label{A-odd} \\
\sum_{l=0}^{\frac{n-1}{2}}\beta_{l}(X(\RR))=&\frac12\beta_*, \label{B-odd}\\
\sum_{l=0}^{\frac{n-1}{2}}\sum_{i=0}^{2l-1}\beta_{i}(X(\RR))=&
\frac{n-1}{4}\beta_* \label{C-odd}.
\end{align}
\qed
\end{lemma}

\begin{lemma}
\label{3together-even}
\cite[Lemma 2.8]{hypersurfaces}
Let $X$ be a maximal real nonsingular projective variety of even dimension $n$. 
Then the following relations hold:
\begin{align}
\sum_{l=1}^{\frac{n}{2}}\sum_{\substack{a+b=2l-1\\a<b}}\beta_a(X(\RR))\beta_b(X(\RR))
=&\,\frac{1}{2}\beta_{\rm even}(X(\RR))\beta_{\rm odd}(X(\RR)) \,,\label{A-even} \\
\sum_{l=1}^{\frac{n}{2}}\sum_{i=0}^{2l-2}\beta_{i}(X(\RR))
=&\,\frac{n}{4}\beta_*-\frac12\beta_{\rm odd}(X(\RR)) \label{B-even}.
\end{align}
\qed
\end{lemma}


\section{Topology of the real Hilbert square}
\label{cut-paste}


A cut-and-paste decomposition of the Hilbert square was introduced for 
nonsingular surfaces in \cite{loss-surfaces} and extended to nonsingular 
varieties of higher dimensions in \cite{hypersurfaces}. Since this decomposition 
is central to the present work, we recall it for the reader's convenience.

\smallskip

We begin by recalling the blowup model of the Hilbert square.  Let $X$ be a
nonsingular complex variety, let $\Delta\subset X\times X$ be the diagonal, and 
denote by \(\tau\) the involution that exchanges the two factors.  Since $\Delta$
is $\tau$-invariant, the involution lifts to
$$
 \Bl(\tau)\colon \Bl_{\Delta}(X\times X)\longrightarrow
 \Bl_{\Delta}(X\times X).
$$
The quotient by this lifted involution is naturally identified with $X^{[2]}$:
$$
 X^{[2]}\cong \Bl_{\Delta}(X\times X)/\Bl(\tau).
$$
The branch divisor of the resulting double covering will be denoted by
$E\subset X^{[2]}$.  The normal bundle of $\Delta$ in $X\times X$ is
canonically isomorphic to $TX$; consequently,
$$
 E\cong \PP(T^*X).
$$
Here, we follow the Grothendieck convention used in \cite{square}: $\PP(V)$
parametrizes the hyperplanes in $V$.  Thus, one may view a point of $E$ as a
point of $X$ together with a tangent direction at that point.  By contrast, a
point of $X^{[2]}\setminus E$ is an unordered pair of distinct points of
$X$.  We shall also use the fact that the normal bundle of $E$ in
$X^{[2]}$ is the square of the tautological line bundle on $\PP(T^*X)$.


\subsection{Decomposition of the real locus}


Suppose now that $X$ is defined over $\RR$, and write $\Conj$ for its real
structure.  The preceding construction is defined over $\RR$, and hence
induces a real structure on $X^{[2]}$.  Away from \(E\), an unordered pair
\(\{x,y\}\) is real precisely when it is preserved by \(\Conj\).  There are two
possibilities:
\begin{enumerate}
\item \(y=\Conj(x)\), with \(x\notin X(\RR)\);
\item both \(x\) and \(y\) belong to \(X(\RR)\).
\end{enumerate}
On \(E\), conjugation sends a point and a tangent direction to their respective
conjugates.  It follows that
$$
 E(\RR)\cong \PP_{\RR}\bigl(T^*X(\RR)\bigr).
$$

If $X(\RR)=\varnothing$, every real point of $X^{[2]}$ is represented by a
conjugate pair, and therefore
$$
 X^{[2]}(\RR)\cong X/\Conj.
$$
Assume from now on that the real locus is nonempty, and write
$$
 X(\RR)=F_1\sqcup\cdots\sqcup F_r
$$
for its decomposition into connected components.  The real points outside
\(E(\RR)\) then split as
\begin{equation}
\label{open-real-strata}
\begin{split}
 X^{[2]}(\RR)\setminus E(\RR)
 ={}&\bigl((X/\Conj)\setminus X(\RR)\bigr)\\
 &\sqcup\bigsqcup_{i=1}^{r}
       \bigl(F_i^{(2)}\setminus\Delta F_i\bigr)
 \sqcup\bigsqcup_{1\leq i<j\leq r}(F_i\times F_j).
\end{split}
\end{equation}
The three terms correspond, respectively, to conjugate nonreal pairs, pairs of
distinct points lying in the same component of \(X(\RR)\), and pairs whose
points lie in different components.

The first two types of strata acquire a common boundary when two points come
together.  More precisely, let \(\HH_0\) be the closure of
\((X/\Conj)\setminus X(\RR)\) in \(X^{[2]}(\RR)\), and, for \(1\leq i\leq r\),
let \(\HH_i\) be the closure of
\(F_i^{(2)}\setminus\Delta F_i\).  These are compact manifolds with boundary,
and
\begin{align}
\label{boundaries-H}
 \partial\HH_0&=E(\RR),
 &\interior\HH_0&\cong (X/\Conj)\setminus X(\RR), \notag\\
 \partial\HH_i&=\PP_{\RR}(T^*F_i),
 &\interior\HH_i&\cong F_i^{(2)}\setminus\Delta F_i,
 \qquad 1\leq i\leq r.
\end{align}
In particular,
\[
 E(\RR)=\bigsqcup_{i=1}^{r}\PP_{\RR}(T^*F_i),
\]
and \(\HH_i\) is attached to \(\HH_0\) along the \(i\)-th component
\(\PP_{\RR}(T^*F_i)\) of this hypersurface.  The pieces \(F_i\times F_j\), on
the other hand, do not meet \(E(\RR)\), since points lying in distinct
components cannot collide.

We have therefore obtained the decomposition
\begin{equation}
\label{connected-components}
 X^{[2]}(\RR)
 =X^{[2]}_{\mathrm{main}}(\RR)
  \bigsqcup X^{[2]}_{\mathrm{extra}}(\RR),
 \qquad
 X^{[2]}_{\mathrm{extra}}(\RR)
 =\bigsqcup_{1\leq i<j\leq r}(F_i\times F_j),
\end{equation}
where
\begin{equation}
\label{main-component-gluing}
 X^{[2]}_{\mathrm{main}}(\RR)
 =\HH_0\cup_{E(\RR)}
   \left(\bigsqcup_{i=1}^{r}\HH_i\right)
 =\bigcup_{i=0}^{r}\HH_i.
\end{equation}
Thus, $E(\RR)$ separates the main component into the $r+1$ manifolds with
boundary $\HH_0,\HH_1,\ldots,\HH_r$.  Figure~\ref{fig:real-hilbert-square-gluing}
summarizes this cut-and-paste description.

\begin{figure}[ht]
\centering
\begin{tikzpicture}[
  x=0.64cm,
  y=0.56cm,
  line cap=round,
  line join=round,
  main arc/.style={draw=black,line width=1.20pt},
  piece arc/.style={draw=black,line width=1.00pt},
  boundary/.style={draw=red!65!black,line width=1.25pt},
  omitted/.style={draw=black!65,line width=0.90pt,densely dotted},
  region label/.style={font=\footnotesize},
  boundary label/.style={font=\scriptsize,text=red!65!black}
]
  \path[fill=blue!3]
    (-6,0)
    arc[start angle=180,end angle=0,radius=6]
    arc[start angle=0,end angle=-180,radius=1]
    arc[start angle=0,end angle=180,radius=2]
    arc[start angle=0,end angle=-180,radius=1]
    arc[start angle=0,end angle=180,radius=1]
    arc[start angle=0,end angle=-180,radius=1]
    -- cycle;

  \draw[main arc]
    (-6,0) arc[start angle=180,end angle=0,radius=6];
  \draw[piece arc]
    (-6,0) arc[start angle=180,end angle=360,radius=1];
  \draw[piece arc]
    (-2,0) arc[start angle=0,end angle=180,radius=1];
  \draw[piece arc]
    (-2,0) arc[start angle=180,end angle=360,radius=1];
  \draw[piece arc]
    (4,0) arc[start angle=180,end angle=360,radius=1];

  \draw[omitted]
    (4,0) arc[start angle=0,end angle=180,radius=2];
  \draw[boundary,densely dotted] (1,0)--(3,0);

  \draw[boundary] (-6,0)--(-4,0);
  \draw[boundary] (-2,0)--(0,0);
  \draw[boundary] (4,0)--(6,0);

  \node[region label] at (-4.65,5.35)
    {\(X^{[2]}_{\mathrm{main}}(\RR)\)};
  \node[region label] at (0,3.05) {\(\HH_0\)};

  \node[region label] at (-5,-0.58) {\(\HH_1\)};
  \node[region label] at (-1,-0.58) {\(\HH_2\)};
  \node[region label] at (2,0.62) {\(\cdots\)};
  \node[region label] at (5,-0.58) {\(\HH_r\)};

  \node[boundary label,above=2pt] at (-5,0)
    {\(\PP_{\RR}(T^*F_1)\)};
  \node[boundary label,above=2pt] at (-1,0)
    {\(\PP_{\RR}(T^*F_2)\)};
  \node[boundary label,above=2pt] at (5,0)
    {\(\PP_{\RR}(T^*F_r)\)};

  \node[region label] at (9.55,1.65)
    {\(X^{[2]}_{\mathrm{extra}}(\RR)\)};
  \node[region label] at (7.05,0) {\(\sqcup\)};
  \node[region label] at (8.45,0)
    {\(\textstyle\bigsqcup_{1\leq i<j\leq r}\)};
  \draw[piece arc,fill=gray!5]
    (11.35,0) ellipse[x radius=1.75,y radius=0.88];
  \node[region label] at (11.35,0) {\(F_i\times F_j\)};
\end{tikzpicture}

\caption{The pieces \(\HH_i\) are attached to \(\HH_0\) along the components
of \(E(\RR)\). The products \(F_i\times F_j\) form separate connected
components.}
\label{fig:real-hilbert-square-gluing}
\end{figure}
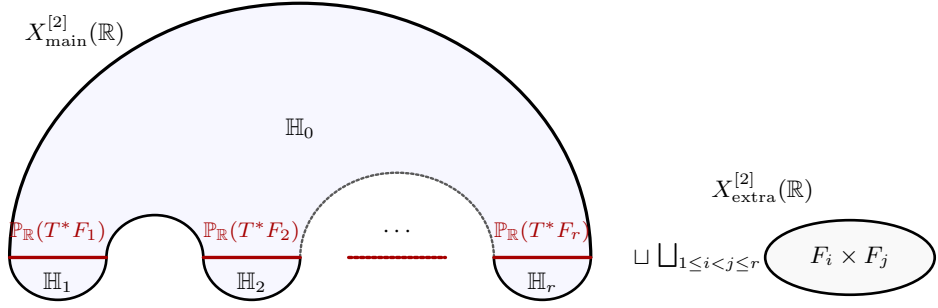

For later use, set
$$
 \HH=\bigsqcup_{i=1}^{r}\HH_i;
 \qquad\text{then}\qquad
 \partial\HH=E(\RR).
$$
Let
$$
 \inc_0\colon E(\RR)=\partial\HH_0\longrightarrow\HH_0,
 \qquad
 \inc\colon E(\RR)=\partial\HH\longrightarrow\HH
$$
be the boundary inclusions.  The induced maps in degree $k$ will be denoted
by
$$
 \inc_0^k\colon H^k(\HH_0)\longrightarrow H^k(E(\RR)),
 \qquad
 \inc^k\colon H^k(\HH)\longrightarrow H^k(E(\RR)).
$$
We use $\inc_0^*$ and $\inc^*$ for the corresponding maps on total
cohomology, and use subscripts for the maps induced in homology.
Finally, we put
$$
 \mu=(\inc_0,\inc)\colon
 E(\RR)\longrightarrow \HH_0\sqcup\HH,
$$
with the same convention for the Mayer--Vietoris maps induced by $\mu$.


\subsection{Betti numbers and boundary maps}


In this subsection, we recall from \cite{hypersurfaces} the formulas for the 
Betti numbersof each $\HH_i,\, 0\le i\le r,$ and the ranks of the corresponding 
maps induced in homology and cohomology. For the rest of the paper, we 
abbreviate $\beta_k(X(\CC))$ to $\beta_k$, and similarly abbreviate  
$\beta_*(X(\CC))$ to $\beta_*$.

\begin{lem}
\label{dimensions-H0}
\cite[Lemma 3.1]{hypersurfaces}
If $X$ is a maximal $n$-dimensional real nonsingular projective variety, then 
the following formulas hold:
\begin{itemize}
\item[ 1)] $\displaystyle \beta_k(\HH_0)=\sum_{i=2n-k}^{2n} \beta_i,$ for every 
integer $0\leq k\leq n-1.$
\item[ 2)] $\displaystyle \beta_*(\HH_0)=\frac{n}{2}\beta_*.$
\end{itemize}
\qed
\end{lem}

\begin{lem}
\label{inc-zero}
\cite[Lemma 3.2]{hypersurfaces}
If $X$ is a maximal $n$-dimensional real nonsingular projective variety, then
$$
\rank (\inc_{0}^{k})=\sum_{i=0}^{k}\beta_i
$$
for every $k<n$.
\qed
\end{lem}

\smallskip

Let $F$ be a compact $C^{\infty}$ manifold of real dimension $m.$ The
complement of the diagonal in its symmetric square $F^{(2)}\setminus \Delta F$ 
is naturally identified with the interior of a smooth compact
$2m$-dimensional manifold $\HH_F$ with boundary $\PP_\RR(T^* F).$ Let
$$
\inc^k_F: H^k(F^{(2)}\setminus \Delta F)=H^k(\HH_F)\ra H^k(\partial\HH_F)
=H^k(\PP_\RR(T^* F))
$$ 
be the restriction homomorphism, and let 
$b\in H^1(F^{(2)}\setminus \Delta F)$ be the class of the double cover 
$g:(F\times F)\setminus \Delta F\ra F^{(2)}\setminus \Delta F.$

\begin{thm}
\label{totaro-basis}
\cite [Theorem 3.3]{hypersurfaces}
Let $z_0, \dots, z_s$ be a basis for $H^*(F)$ and let $Z_i$ be a closed 
pseudomanifold in $F$ that represents the class $z_i.$ For every integer 
$k\geq 0,$ we have:
\begin{itemize}
\item[ 1)] A basis for $H^k(F^{(2)}\setminus \Delta F)$  is given by the elements 
$g_*(z_i\otimes z_j)$ with $\deg z_i +\deg z_j=k$ and $i < j$, together with the
 elements $b^j[Z_i^{(2)}\setminus  \Delta Z_i],$ where $2\deg z_i+j=k,\,i\geq 0$, 
 and $0\leq j\leq m-1-\deg z_i.$
\item[ 2)] $\inc_F^k\left(g_*(z_i\otimes z_j)\right)=0$ for  all $i<j.$
\item[ 3)] The restrictions $\inc_F^k(b^j[Z_i^{(2)}\setminus  \Delta Z_i])$ satisfying 
$2\deg z_i+j=k,\,i\geq 0$  and $0\leq j\leq m-1-\deg z_i$ form a basis of   
$\im(\inc_F^k)\subseteq H^k\left(\PP_\RR(T^* F)\right).$
 \end{itemize}
\qed
\end{thm}

Applying the theorem to our setting yields the following corollaries.

\begin{cor}
\label{beta-square}
If $X$ is an $n$-dimensional nonsingular real projective variety, then, for 
every $i=1,\dots,r$ and $k\geq 0,$ the following formulas hold:
\begin{itemize}
\item[ 1)] $\displaystyle \beta_{2k}(\HH_i)=\sum_{\substack{a+b=2k\\a<b}} 
\beta_a(F_i)\beta_b(F_i)+\frac12\beta_k(F_i)(\beta_k(F_i)-1)
+\sum_{l=2k-n+1}^{k} \beta_l(F_i).$
\item[ 2)] $\displaystyle \beta_{2k+1}(\HH_i)=\sum_{\substack{a+b=2k+1\\a<b}} 
\beta_a(F_i)\beta_b(F_i)+\sum_{l=2k+1-n+1}^{k} \beta_l(F_i).$
\item[ 3)] $\displaystyle \beta_*(\HH_i)=\frac{1}2\beta_*(F_i)(\beta_*(F_i)-1)
+\sum_{k=0}^n (n-k)\beta_k(F_i).$
\item[ 4)] If, in addition, $X$ is maximal, then $\displaystyle \beta_*(\HH)=
\frac12\sum_{i=1}^r\beta_*^2(F_i)+\frac{n-1}{2}\beta_*.$ 
\end{itemize}
\qed
\end{cor}

\begin{cor}
\label{inc-sym}
If $X$ is an  $n$-dimensional real projective manifold, then
\begin{equation*}
\rank(\inc^m)=\sum_{[\frac{m}{2}]\geq k\geq m-n+1}\beta_k(X(\RR)),
\end{equation*}
for every $m\in\{0,\dots,  2n\}.$
\qed
\end{cor}

We conclude this section by recalling the following proposition, which  generalizes 
Proposition 3.2 in \cite{loss-surfaces}. Its proof extends verbatim to higher dimensions 
and will be omitted.

\begin{prop}
\label{calculus}
If $X$ is an $n$-dimensional nonsingular projective variety defined over $\RR$,
then  
\begin{itemize}
\item[ 1)] The following identity holds 
\begin{equation}
\label{M-chi}
\chi(X^{[2]}(\RR))=\frac12\beta_*-\beta_{\rm odd}+\frac12\chi(X(\RR))^2-\chi(X(\RR)).
\end{equation}
\item[ 2)] If $\Tors_2 H_*(X;\ZZ)=0,$ then
the following formula holds
\begin{equation}
\label{totaro-notorsion}
\beta_*(X^{[2]})=\frac12\beta_*(\beta_*-1)+n\beta_*-\beta_{\rm odd}.
\end{equation}
\item[ 3)] If $\Tors_2 H_*(X;\ZZ)\neq 0,$ then the following inequality holds
\begin{equation}
\label{totaro-torsion}
\beta_*(X^{[2]})\geq \frac12\beta_*(\beta_*-1)+n\beta_*-\beta_{\rm odd}.
\end{equation}
\end{itemize}
\qed
\end{prop}


\section{The Smith--Thom deficiency via Mayer--Vietoris}
\label{deficiency-mv-section}


Let $X$ be a Smith--Thom maximal nonsingular projective variety such that  
$\Tors_2 H_*(X;\ZZ)=0$. 

\smallskip

The geometric decomposition of the real locus of the Hilbert square $X^{[2]}$
allows us to compute its Smith--Thom deficiency using the Mayer--Vietoris sequence,
following an approach similar to that used in
\cite{loss-surfaces, hypersurfaces}. The results depend on the parity of
$n=\Dim_\CC X\geq 2$.


\subsection{The even-dimensional formula} 


We assume here that the dimension $n$ of $X$ is even. Using
\eqref{M-chi} and \eqref{totaro-notorsion}, the deficiency of 
$X^{[2]}$ can be computed as follows:
\begin{align*}
\defi(X^{[2]})=&\, \beta_*(X^{[2]})-\beta_*(X^{[2]}(\RR))\notag\\ 
=&\, \beta_*(X^{[2]})-2\beta_{\rm odd}(X^{[2]}(\RR)) - \chi(X^{[2]}(\RR))\notag\\ 
=&\, \frac12\beta_*^2+(n-1)\beta_* - 2\beta_{\rm odd}(X^{[2]}(\RR)) 
-\frac12 \chi(X(\RR))^2+\chi(X(\RR)).
\end{align*}
Since $X$ is maximal, we have $\chi(X(\RR))=\beta_*-2\beta_{\rm odd}(X(\RR))$, 
and we find 
\begin{align}
\label{deficiency-square-even}
\defi(X^{[2]})=&\, n\beta_*+2\beta_*\beta_{\rm odd}(X(\RR))-2\beta_{\rm odd}^2(X(\RR))
-2\beta_{\rm odd}(X(\RR))\notag \\
&- \, 2\beta_{\rm odd}(X^{[2]}(\RR))\notag \\
=&\, n\beta_*+2\beta_{\rm even}(X(\RR))\beta_{\rm odd}(X(\RR))
-2\beta_{\rm odd}(X(\RR))\notag \\
&- \, 2\beta_{\rm odd}(X^{[2]}(\RR)).
\end{align}
Notice now from (\ref{connected-components}) that 
\begin{align}
\label{sum-betti-odd}
\beta_{\rm odd}(X^{[2]}(\RR))=&\, \beta_{\rm odd}(X^{[2]}_{\rm main}(\RR))
+\beta_{\rm odd}(X^{[2]}_{\rm extra}(\RR))\notag \\
=&\, 2 \sum_{l=1}^{\frac{n}2}\left(\beta_{2l-1}(X^{[2]}_{\rm main}(\RR))
+\beta_{2l-1}(X^{[2]}_{\rm extra}(\RR))\right).
\end{align}
To compute $\beta_{2l-1}(X^{[2]}_{\rm main}(\RR))$ for every integer $l$ 
such that $1\leq l \leq \frac{n}2,$ we use the Mayer--Vietoris sequence
\begin{align*}
\cdots  \ra H_{2l-1}(E(\RR))\xrightarrow{\mu_{2l-1}}\bigoplus_{i=0}^r & H_{2l-1}(\HH_i)
\ra H_{2l-1}(X^{[2]}_{\text{main}}(\RR))\\
\ra &H_{2l-2}(E(\RR))\xrightarrow{\mu_{2l-2}} \dots
\end{align*}
which gives rise to the short exact sequence
$$
0\ra\coker (\mu_{2l-1})\ra H_{2l-1}(X^{[2]}_{\text{main}}(\RR))\ra \Ker(\mu_{2l-2})\ra 0.
$$
Therefore, for any $1\leq l\leq n,$ we have  
\begin{align}
\label{master-main-even}
\beta_{2l-1}(X^{[2]}_{\text{main}}(\RR))=&\, \Dim\coker (\mu_{2l-1})+\Dim\Ker(\mu_{2l-2})\notag \\
=&\, 
\sum_{i=0}^r\beta_{2l-1}(\HH_i)+\beta_{2l-2}(E(\RR))\notag \\
&\, -\rank(\mu_{2l-1})-\rank(\mu_{2l-2}).
\end{align}
Using the decomposition of $X^{[2]}_{\rm extra}(\RR)$ in~\eqref{connected-components}, 
we also notice that
\begin{align}
\label{extra-odd}
\beta_{2l-1}(X^{[2]}_{\rm extra}(\RR))=&\, \sum_{1\leq s<t\leq r} \beta_{2l-1}(F_s\times F_t) \notag \\
=&\, \sum_{i+j=2l-1}\sum_{1\leq s<t\leq r}\beta_i(F_s)\beta_j(F_t).
\end{align}

For every integer $l$ such that $1\leq l\leq \frac{n}{2},$  a direct computation 
using (\ref{extra-odd}) and Corollary \ref{beta-square} yields

\begin{align}
\label{single-quadratic}
&\, \beta_{2l-1}(X^{[2]}_{\rm extra}(\RR))+ \sum_{i=1}^r\beta_{2l-1}(\HH_i)\notag \\
=&\sum_{i+j=2l-1}\sum_{1\leq s<t\leq r}\beta_i(F_s)\beta_j(F_t) 
+\sum_{i=1}^r\sum_{\substack{a+b=2l-1\\a<b}} \beta_a(F_i)\beta_b(F_i)+
 \sum_{i=1}^r\sum_{c=0}^{l-1}\beta_c(F_i)\notag \\
=&\sum_{\substack{a+b=2l-1\\a<b}}\beta_a(X(\RR))\beta_b(X(\RR))  
+ \sum_{c=0}^{l-1}\beta_c(X(\RR)).
\end{align}
Notice from Lemma \ref{dimensions-H0} that for every integer $1\leq l\leq \frac{n}{2}$ we have
\begin{equation}\label{betaH0}
\beta_{2l-1}(\HH_0)=\sum_{i=2n-2l+1}^{2n} \beta_i=\sum_{j=0}^{2l-1} \beta_j.
\end{equation}

We now infer that

\begin{align}
\label{beta_odd-mu}
\beta_{2l-1}(X^{[2]}(\RR))=&\sum_{j=0}^{2l-1} \beta_j 
+\sum_{\substack{a+b=2l-1\\a<b}}\beta_a(X(\RR))\beta_b(X(\RR)) \notag\\ 
 &+ \sum_{c=0}^{l-1}\beta_c(X(\RR))+\sum_{i=0}^{2l-2}\beta_i(X(\RR))\notag \\
 &\, -\rank(\mu_{2l-1})-\rank(\mu_{2l-2}).
\end{align}
Therefore, we obtain
\begin{align}
\label{sum_beta_odd_mu}
\frac12\beta_{\rm odd}(X^{[2]}(\RR))=&\sum_{l=1}^{\frac{n}{2}}\sum_{j=0}^{2l-1} \beta_j 
+\sum_{l=1}^{\frac{n}{2}}\sum_{\substack{a+b=2l-1\\a<b}}\beta_a(X(\RR))\beta_b(X(\RR))
+\sum_{l=1}^{\frac{n}{2}} \sum_{c=0}^{l-1}\beta_c(X(\RR)) \notag\\ 
 &+\sum_{l=1}^{\frac{n}{2}}\sum_{i=0}^{2l-2}\beta_i(X(\RR))
 -\sum_{l=1}^{\frac{n}{2}}\left(\rank(\mu_{2l-1})+\rank(\mu_{2l-2})\right)\notag \\
 =& \sum_{l=1}^{\frac{n}{2}}\sum_{j=0}^{2l-1} \beta_j +\sum_{l=1}^{\frac{n}{2}} \sum_{c=0}^{l-1}\beta_c(X(\RR)) 
 +\frac{n}{4}\beta_*-\frac12\beta_{\rm odd}(X(\RR))\notag \\ 
 + &\,\frac{1}{2}\beta_{\rm even}(X(\RR))\beta_{\rm odd}(X(\RR))
 -\sum_{k=0}^{n-1}\rank(\mu_{k}).
\end{align}

As a consequence, we find:

\begin{thm}
\label{even-rk-thm}
Let $X$ be a nonsingular real projective variety of even dimension $n\geq 2.$ 
If $X$ is Smith--Thom maximal and $\Tors_2 H_*(X;\ZZ)=0,$ then the 
Smith--Thom deficiency of its Hilbert square is as follows:
\begin{equation}
\label{deficiency-even_mu}
\defi(X^{[2]})=
\, 4\left( \sum_{k=0}^{n-1}\rank(\mu_{k})- \sum_{l=1}^{\frac{n}{2}}\sum_{j=0}^{2l-1} \beta_j 
-\sum_{l=1}^{\frac{n}{2}} \sum_{c=0}^{l-1}\beta_c(X(\RR))\right).
\end{equation}
\end{thm}
\proof
Substitute (\ref {sum_beta_odd_mu}) into (\ref{deficiency-square-even}).
\qed

\begin{cor}
\label{rank-defect-even}
If $X$ is a nonsingular Smith--Thom maximal variety of even dimension $n=\Dim X\geq 2$ and 
$\Tors_2 H_*(X;\ZZ)=0,$ then $X^{[2]}$ is maximal if and only if
$$
\sum_{k=0}^{n-1}\rank(\mu_{k}) = 
\sum_{l=1}^{\frac{n}{2}}\sum_{j=0}^{2l-1} \beta_j +
\sum_{l=1}^{\frac{n}{2}} \sum_{c=0}^{l-1}\beta_c(X(\RR)).
$$
\qed
\end{cor}
\begin{rmk}
\label{surface-case}
{\rm Since  $\rank(\mu_{0})=\beta_0(X(\RR))$,   Corollary \ref{rank-defect-even} 
shows that, for surfaces, the Hilbert square $X^{[2]}$ is maximal if and only if 
$\rank (\mu_1)= 1+ \beta_1$. Note also that, as shown in \cite[Lemma 6.1]{loss-surfaces}, 
$\rank(\mu_1)= 1+ \beta_1$ whenever $X$ is a maximal surface and $X(\RR)$ is 
connected. Combining these two observations, we proved in \cite{loss-surfaces} that, 
for a maximal surface, the conditions $\beta_0(X(\RR))=1$ and 
$\Tors_2 H_*(X;\ZZ)=0$ are sufficient for the maximality of $X^{[2]}$.}
\end{rmk}


\subsubsection{Proof of Theorem \ref{defect-mv} for even-dimensional varieties}


Recall that $\displaystyle \mu =(\inc_{0},\inc) : E(\RR)\ra \HH_0\sqcup \HH.$ 
The induced map in cohomology is 
$$
\mu^k:H^k(\HH_0) \oplus H^k(\HH)\ra H^k(E(\RR)) 
$$  
given by $\mu^k=\inc_0^k+\inc^k,$ where the maps $\inc_0^k$ and $\inc^k$ are 
the restrictions 
$$
\inc_{0}^k:  H^k(\HH_0)\ra H^k(E(\RR))\quad{\text{and}}
\quad  \inc^{k}: H^k(\HH)\ra H^k(E(\RR)),
$$ 
respectively.
In particular, we find $\im(\inc^{k})\subseteq\im(\mu^k),$ and 
$\im(\inc_0^{k})\subseteq\im(\mu^k),$ 
for every $0\leq k\leq n-1.$ Since all groups are finite-dimensional vector spaces 
over $\FF_2,$ the maps $\mu_k$ and $\mu^k$ are dual and therefore have the 
same rank. Hence, 
\begin{equation}
\label{bound-ranks}
\rank (\mu_k)\geq \max \{\rank (\inc_0^k),\rank (\inc^k)\}.
\end{equation}
Substituting the formulas from Lemma \ref{inc-zero} and Corollary \ref{inc-sym} 
into the deficiency formula \eqref{deficiency-even_mu} and using \eqref{bound-ranks}
gives the following lower bound for the Smith--Thom deficiency of $X^{[2]}$:
\begin{align*}
\frac14\defi(X^{[2]})\geq &\,
 \max\left\{\sum_{k=0}^{n-1}\sum_{i=0}^{k}\beta_i\, ,\, 
 \sum_{k=0}^{n-1}\sum_{i=0}^{[k/2]}\beta_i(X(\RR))\right\} 
  - \sum_{l=1}^{\frac{n}{2}}\sum_{j=0}^{2l-1} \beta_j 
  -\sum_{l=1}^{\frac{n}{2}} \sum_{c=0}^{l-1}\beta_c(X(\RR))\\
\geq &\, \sum_{k=0}^{n-1}\sum_{i=0}^{[k/2]}\beta_i(X(\RR))
- \sum_{l=1}^{\frac{n}{2}}\sum_{j=0}^{2l-1} \beta_j 
-\sum_{l=1}^{\frac{n}{2}} \sum_{c=0}^{l-1}\beta_c(X(\RR))\\
=&\,\sum_{l=1}^{\frac{n}{2}} \sum_{c=0}^{l-1}\beta_c(X(\RR)) 
- \sum_{l=1}^{\frac{n}{2}}\sum_{j=0}^{2l-1} \beta_j.
\end{align*}
\qed

In particular, we obtain the following criterion:

\begin{cor}
\label{RealDefectEstimate-even}
Let $X$ be a nonsingular maximal variety with $n=\Dim\, X$ even
and $\Tors_2 H_*(X;\ZZ)=0.$ If
\begin{equation*}
\sum_{k=1}^{\frac{n}2} \sum_{i=0}^{k-1}\beta_i(X(\RR))>
 \sum_{l=1}^{\frac{n}{2}}\sum_{j=0}^{2l-1} \beta_j,
\end{equation*}
then $X^{[2]}$ is not maximal. \qed
\end{cor}


\subsection{The odd-dimensional formula}


Assume that $n=\Dim X\geq 3$ is odd and  $\Tors_2 H_*(X;\ZZ)=0$. Using
\eqref{M-chi}, \eqref{totaro-notorsion}, and $\chi(X(\RR))=0,$
the deficiency of $X^{[2]}$ can be computed as follows:
\begin{align}
\label{deficiency-square-odd}
\defi(X^{[2]})=&\, \beta_*(X^{[2]})-\beta_*(X^{[2]}(\RR))\notag\\ 
=&\, \beta_*(X^{[2]})-2\beta_{\rm even}(X^{[2]}(\RR))+\chi(X^{[2]}(\RR))\notag\\ 
=&\, \frac12\beta_*^2+n\beta_*-2\beta_{\rm odd} - 2\beta_{\rm even}(X^{[2]}(\RR)).
\end{align}
Notice now that 
\begin{align}
\label{sum-betti-even}
\beta_{\rm even}(X^{[2]}(\RR))=&\, \beta_{\rm even}(X^{[2]}_{\rm main}(\RR))
+\beta_{\rm even}(X^{[2]}_{\rm extra}(\RR))\notag \\
=&\, 2 \sum_{l=0}^{\frac{n-1}2}\left(\beta_{2l}(X^{[2]}_{\rm main}(\RR)) 
+ \beta_{2l}(X^{[2]}_{\rm extra}(\RR))\right).
\end{align}
To compute $\beta_{2l}(X^{[2]}_{\rm main}(\RR))$ for every integer $l$ 
such that $0\leq 2l\leq n-1,$ we use the Mayer--Vietoris sequence
$$
\cdots  \ra H_{2l}(E(\RR))\xrightarrow{\mu_{2l}}\bigoplus_{i=0}^r H_{2l}(\HH_i)
\ra H_{2l}(X^{[2]}_{\text{main}}(\RR))\ra H_{2l-1}(E(\RR))\xrightarrow{\mu_{2l-1}} \dots
$$
which gives rise to the short exact sequence
$$
0\ra\coker (\mu_{2l})\ra H_{2l}(X^{[2]}_{\text{main}}(\RR))\ra \Ker(\mu_{2l-1})\ra 0,
$$
and so
\begin{align}
\label{master-main}
\beta_{2l}(X^{[2]}_{\text{main}}(\RR))=&\, \Dim\coker (\mu_{2l})
+\Dim\Ker(\mu_{2l-1})\notag \\
=&\,\sum_{i=0}^r\beta_{2l}(\HH_i)+\beta_{2l-1}(E(\RR)) -\rank(\mu_{2l})-\rank(\mu_{2l-1}).
\end{align}
Therefore, for every integer $l$ such that $0\leq 2l\leq n-1,$  we have
\begin{align}
\label{beta-2l-first}
\beta_{2l}(X^{[2]}(\RR))=\,& \beta_{2l}(X^{[2]}_{\rm main}(\RR))+
\beta_{2l}(X^{[2]}_{\rm extra}(\RR))\notag\\
=&\, \beta_{2l}(X^{[2]}_{\rm extra}(\RR))+
\sum_{i=0}^r\beta_{2l}(\HH_i)+\beta_{2l-1}(E(\RR))\notag \\
&\, -\rank(\mu_{2l})-\rank(\mu_{2l-1}).
\end{align}
Notice that by the K\"unneth formula we have 
\begin{align}
\label{extra-even}
\beta_{2l}(X^{[2]}_{\rm extra}(\RR))
=&\, \sum_{1\leq s<t\leq r} \beta_{2l}(F_s\times F_t) \notag \\
=&\, \sum_{i+j=2l}\sum_{1\leq s<t\leq r}\beta_i(F_s)\beta_j(F_t).
\end{align}

For every integer $l$ such 
that $0\leq 2l\leq n-1,$  a direct computation using (\ref{extra-even}) 
and Corollary \ref{beta-square} shows that

\begin{align}
\label{even-single-quadratic}
\beta_{2l}(X^{[2]}_{\rm extra}(\RR))+\sum_{i=1}^r\beta_{2l}(\HH_i)=&\, 
\sum_{i+j=2l}\sum_{1\leq s<t\leq r}\beta_i(F_s)\beta_j(F_t) 
 +\sum_{i=1}^r\sum_{\substack{a+b=2l\\a<b}} \beta_a(F_i)\beta_b(F_i)\notag \\
&+\, \sum_{i=1}^r\frac12\beta_l(F_i)(\beta_l(F_i)-1)
+\sum_{i=1}^r\sum_{a=0}^{l} \beta_a(F_i)\notag\\
=&\, \frac12\sum_{a+b=2l}\beta_a(X(\RR))\beta_b(X(\RR))  \notag\\ 
&\, +\sum_{a=0}^{l}\beta_a(X(\RR))-\frac12\beta_l(X(\RR)).
\end{align}
According to Lemma \ref{dimensions-H0}, we have 
$$
\beta_{2l}(\HH_0)=\sum_{i=2n-2l}^{2n}\beta_i(X)=\sum_{j=0}^{2l}\beta_j.
$$
Thus, by \eqref{even-single-quadratic}, equation
\eqref{beta-2l-first} can be written as follows:
\begin{align}
\label{beta-2l}
\beta_{2l}(X^{[2]}(\RR))= & \sum_{j=0}^{2l}\beta_j
+\frac12\sum_{i+j=2l}\beta_i(X(\RR))\beta_j(X(\RR))+ 
\sum_{a=0}^{l}\beta_a(X(\RR)) - \frac12\beta_l(X(\RR))
\notag\\
&\,  + \sum_{i=0}^{2l-1}\beta_i(X(\RR))-\rank(\mu_{2l})-\rank(\mu_{2l-1}).
\end{align}
From Lemma \ref{elementary-odd} and formula (\ref{beta-2l}) we find:

\begin{align}
\label{sum_beta_even_mu}
\beta_{\rm even}(X^{[2]}(\RR)) 
=&\,2\sum_{l=0}^{\frac{n-1}{2}}\beta_{2l}(X^{[2]}(\RR))\notag\\
=&\, 2\sum_{l=0}^{\frac{n-1}{2}} \sum_{j=0}^{2l}\beta_j
+\sum_{l=0}^{\frac{n-1}{2}}\sum_{i+j=2l}\beta_i(X(\RR))\beta_j(X(\RR)) 
+2\sum_{l=0}^{\frac{n-1}{2}}\sum_{a=0}^{l-1}\beta_a(X(\RR))\notag\\
&\,+ 2\sum_{l=0}^{\frac{n-1}{2}}\sum_{i=0}^{2l-1}\beta_i(X(\RR))
+\sum_{l=0}^{\frac{n-1}{2}}\beta_l(X(\RR))
-2\sum_{j=0}^{n-1}\rank(\mu_j)\notag\\
=&\,\frac{ \beta^2_*}{4}
+\frac{n\beta_*}{2}+2\sum_{l=0}^{\frac{n-1}{2}} \sum_{j=0}^{2l}\beta_j
+2\sum_{l=1}^{\frac{n-1}{2}}\sum_{i=0}^{l-1}\beta_i(X(\RR))
-2\sum_{j=0}^{n-1}\rank(\mu_j).
\end{align}

As in the even-dimensional case, we find:

\begin{thm}
\label{odd-rk-thm}
Let $X$ be a nonsingular real projective variety of odd dimension $n\geq 3.$ 
If $X$ is Smith--Thom maximal and $\Tors_2 H_*(X;\ZZ)=0,$ then the 
Smith--Thom deficiency of its Hilbert square is as follows:
\begin{equation}
\label{deficiency-odd_mu}
\defi(X^{[2]})=
\, 4\left( \sum_{k=0}^{n-1}\rank(\mu_{k})
-\sum_{l=0}^{\frac{n-1}{2}}\sum_{j=0}^{2l} \beta_j -
\sum_{l=1}^{\frac{n-1}{2}} \sum_{i=0}^{l-1}\beta_i(X(\RR))
-\frac{\beta_{\rm odd}}{2}\right)
\end{equation}
\end{thm}
\proof
Substitute \eqref {sum_beta_even_mu} into \eqref{deficiency-square-odd}.
\qed

\begin{cor}
\label{rank-defect-odd}
If $X$ is a nonsingular Smith--Thom maximal variety with $n=\Dim\, X$ odd and 
$\Tors_2 H_*(X;\ZZ)=0$, then $X^{[2]}$ is Smith--Thom maximal if and only if
$$
\sum_{k=0}^{n-1}\rank(\mu_{k}) = 
\sum_{l=0}^{\frac{n-1}{2}}\sum_{j=0}^{2l} \beta_j +
\sum_{l=1}^{\frac{n-1}{2}} \sum_{i=0}^{l-1}\beta_i(X(\RR))+\frac{\beta_{\rm odd}}2. 
$$
\qed
\end{cor}


\subsubsection{Proof of Theorem \ref{defect-mv} for odd-dimensional varieties}


As in the even-dimensional case, substituting the formulas from 
Lemma \ref{inc-zero} and Corollary \ref{inc-sym} into \eqref{deficiency-odd_mu} 
and applying  \eqref{bound-ranks}, we obtain the lower bound for the 
Smith--Thom deficiency of $X^{[2]}$ asserted in Theorem \ref{defect-mv}:
\begin{align*}
\frac14\defi(X^{[2]})\geq &\,
 \max\left\{\sum_{k=0}^{n-1}\sum_{i=0}^{k}\beta_i, 
 \sum_{k=0}^{n-1}\sum_{i=0}^{[k/2]}\beta_i(X(\RR))\right\} \\
 &\, - \sum_{l=0}^{\frac{n-1}{2}}\sum_{j=0}^{2l} \beta_j -
\sum_{l=1}^{\frac{n-1}{2}} \sum_{i=0}^{l-1}\beta_i(X(\RR))-\frac{\beta_{\rm odd}}{2}\\
\geq &\, \sum_{k=0}^{n-1}\sum_{i=0}^{[k/2]}\beta_i(X(\RR))
-   \sum_{l=0}^{\frac{n-1}{2}}\sum_{j=0}^{2l} \beta_j -
\sum_{l=1}^{\frac{n-1}{2}} \sum_{i=0}^{l-1}\beta_i(X(\RR))-\frac{\beta_{\rm odd}}{2}\\
=&\,\sum_{k=0}^{\frac{n-1}2} \sum_{i=0}^{k}\beta_i(X(\RR))-
\sum_{l=0}^{\frac{n-1}{2}}\sum_{j=0}^{2l} \beta_j -\frac12\beta_{\rm odd}.
\end{align*}
\qed

\smallskip

In particular, we obtain the following criterion:

\begin{cor}
\label{RealDefectEstimate-odd}
Let $X$ be a nonsingular Smith--Thom maximal variety with $n=\Dim\, X$ odd
and $\Tors_2 H_*(X;\ZZ)=0.$ If
\begin{equation*}
\sum_{k=0}^{\frac{n-1}2} \sum_{i=0}^{k}\beta_i(X(\RR))>
 \sum_{l=0}^{\frac{n-1}{2}}\sum_{j=0}^{2l} \beta_j +\frac{\beta_{\rm odd}}2,
\end{equation*}
then $X^{[2]}$ is not Smith--Thom maximal.
\qed
\end{cor}


\section{Applications}
\label{examples}


\subsection{Abelian varieties}


We now apply Theorem \ref{defect-mv} to real abelian varieties. For maximal real abelian 
varieties, the lower bound in Theorem \ref{defect-mv} proves nonmaximality 
of the Hilbert square in every dimension except $2$ and $4$. 
To handle these cases, we improve the estimate for the total rank of the 
Mayer--Vietoris maps in \eqref{deficiency-even_mu} and \eqref{deficiency-odd_mu}.

Let $X$ be an $n$-dimensional Smith--Thom maximal abelian variety. In this case, 
$X(\RR)$ is a disjoint union of $2^n$ tori of real dimension $n.$ Recall that the lower 
bound for the Smith--Thom deficiency obtained in Theorem \ref{defect-mv} depends on 
the parity of $n$ and is given by
\begin{equation*}
D_n:=
\begin{cases}
\displaystyle
4\left(\frac12\beta_{\rm{even}}+\sum_{l=1}^{\frac{n-1}{2}}\sum_{i=0}^{l-1}\beta_i(X(\RR))
-\sum_{l=0}^{\frac{n-1}{2}} \sum_{j=0}^{2l}\beta_j\right),\, \text{if $n$ is odd}\\
\displaystyle
4\left(\sum_{l=1}^{\frac{n}{2}}\sum_{i=0}^{l-1}\beta_i(X(\RR))
-\sum_{l=1}^{\frac{n}{2}}\sum_{j=0}^{2l-1} \beta_j\right),\, \text{if $n$ is even}.
\end{cases}
\end{equation*}

\begin{lem} Let $n\geq 2$ be an integer. Then:
\begin{equation}
\label{LowerBoundDiscrepancy}
D_n=
\begin{cases}
\, 2^{n}(n+1) \binom{n}{\frac{n-1}2}-2^{2n-1}-n\binom{2n}{n},\, {\text{if $n$ is odd,}}\\
\, 2^{n} n \binom{n}{\frac{n}2}-2^{2n-1}-{n}\binom{2n}{n},\, {\text{if $n$ is even.}}\\
\end{cases}
\end{equation}
\end{lem}

\proof First suppose that $n$ is odd. By elementary combinatorics, we find:
\begin{align*}
\frac{D_n}4= 
&\, \sum_{k=0}^{\frac{n-1}2} \sum_{i=0}^{k}\beta_i(X(\RR))-
\sum_{l=0}^{\frac{n-1}{2}}\sum_{j=0}^{2l} \beta_j -\frac12\beta_{\rm odd}\\ 
= &\, 2^n\sum_{k=0}^{\frac{n-1}2} \sum_{i=0}^{k}\binom{n}{i}-
\sum_{l=0}^{\frac{n-1}{2}}\sum_{j=0}^{2l} \binom{2n}{j} -\frac12\sum_{k=1}^{n}\binom{2n}{2k-1}\\ 
=&\, 2^{n-2}(n+1) \binom{n}{\frac{n-1}2}-2^{2n-3}-\frac{n}4\binom{2n}{n}.
\end{align*}

Now suppose that $n$ is even. In this case, we find:
\begin{align*}
\frac{D_n}4= &\, \sum_{k=1}^{\frac{n}2} \sum_{i=0}^{k-1}\beta_i(X(\RR))-
 \sum_{l=1}^{\frac{n}{2}}\sum_{j=0}^{2l-1} \beta_j \\
 =&\,2^n\sum_{k=1}^{\frac{n}2} \sum_{i=0}^{k-1}\binom{n}{i}
 -\sum_{l=1}^{\frac{n}{2}}\sum_{j=0}^{2l-1}\binom{2n}{j}\\
=&\, 2^{n-2} n \binom{n}{\frac{n}2}-2^{2n-3}-\frac{n}4\binom{2n}{n}.
\end{align*}
The conclusion of the lemma follows immediately.
\qed

\medskip

A direct computation shows that $D_2=-4$ and $D_4=-24.$ 
Aside from these two cases, the lower bound $D_n$ 
has a rather uniform behavior, which we describe next.

\begin{lem}
\label{bound+monotonicity} 
The following assertions hold:
\begin{itemize}
\item[ 1)] $D_n\geq2^{2n-\epsilon}$, where $\epsilon=4$ if $n$ is odd and $n\geq 3$, 
while $\epsilon=6$ if $n$ is even and $n\geq 6.$
\item[ 2)] $D_{n+2}>D_n$ for every odd integer $n\geq 3$ and for every even 
integer $n\geq 4.$
\end{itemize}
\end{lem}
\proof The proof is elementary. We assume first that $n$ is even. A direct computation 
shows $D_6=88>0.$ Notice now that    
\begin{align*}
D_{n+2}= &\,2^{n+2}(n+2)\binom{n+2}{\frac{n}2+1}
-2^{2n+3}-(n+2)\binom{2n+4}{n+2}\\
= &\,\frac{16(n+1)}{n}2^{n}n\binom{n}{\frac{n}2}-2^{2n+3}
-\frac{4(2n+1)(2n+3)}{n+1}\binom{2n}{n}\\
= &\,\frac{16(n+1)}{n}D_n+ \left(\frac{16(n+1)}{n}-16\right)2^{2n-1}\\
&+ \left(16(n+1)-\frac{4(2n+1)(2n+3)}{n+1}\right)\binom{2n}{n}\\
=&\,\frac{16(n+1)}{n}D_n+\frac{1}{n}2^{2n+1}+\frac{4}{n+1}\binom{2n}{n}.
\end{align*}
Therefore $D_{n+2}>16D_n,$ and so, by induction, for every 
even integer $n\geq 6$ we have 
$$
D_n\geq D_6 \times 16^{(n-6)/2}=88 \times 2^{2n-12}> 2^{2n-6}.
$$
Also, since $D_{n+2}>16D_n,$ we find that $D_n>0$ for every even integer $n\geq 6,$ 
and $D_{n+2}>D_n$ for every even integer $n\geq 4.$ 

In the case when $n$ is odd, we proceed similarly. We find
\begin{align*}
D_{n+2}= &\, 2^{n+2}(n+3) \binom{n+2}{\frac{n+1}2}-2^{2n+3}
-(n+2)\binom{2n+4}{n+2}\\
= &\,\frac{16(n+2)}{n+1}2^{n}(n+1) \binom{n}{\frac{n-1}2
}-2^{2n+3}-\frac{4(2n+1)(2n+3)}{n+1}\binom{2n}{n}\\
= &\,\frac{16(n+2)}{n+1}D_n+ \left(\frac{16(n+2)}{n+1}-16\right)2^{2n-1}\\
&+ \frac{4}{n+1}\left(4n(n+2)-(2n+1)(2n+3)\right)\binom{2n}{n}\\
=&\,\frac{16(n+2)}{n+1}D_n+\frac{1}{n+1}2^{2n+3}-\frac{12}{n+1}\binom{2n}{n}\\
>&\,\frac{16(n+2)}{n+1}D_n+\frac{4}{n+1}\left[2\binom{2n+1}{n}-3\binom{2n}{n}\right]\\
>&\,\frac{16(n+2)}{n+1}D_n+\frac{4(n-1)}{(n+1)^2}\binom{2n}{n}.
\end{align*}
Therefore, $D_{n+2}>16D_n$ for every odd integer $n\geq 3.$
Since $D_3=4,$ by induction it follows that $D_n$ is positive and $D_{n+2}>D_n$ 
for every odd integer $n\geq 3.$ Furthermore,
$$
D_n\geq D_3\times16^{(n-3)/2}=2^{2n-4}.
$$
\qed


\subsubsection{A refined rank estimate}


Let  $(X, \Conj)$ be a Smith--Thom maximal abelian variety of dimension $n\geq 2$. 
In this case, the action of $\Conj $ on $X$ is diffeomorphic to the action of 
${\mathrm {Id}} \times (-1)$ on $T^n\times T^n,$ where $T^n$ is the $n$-dimensional 
real torus, as follows from the normal form of involutions on  finitely generated free 
$\ZZ$-modules \cite{comessatti}. Thus, $X(\RR)$ is the union of $2^n$ disjoint copies of 
$T^n,$ while $\interior(\HH_0)=X/\Conj\setminus X(\RR)$ is diffeomorphic to the product 
$T^n \times (T^n\setminus\Fix(G))/G,$ where $G$ is the cyclic group of order two acting 
on $T^n$ as multiplication by $-1.$ The fixed-point locus of $G$ is denoted by $\Fix(G)$ 
and consists of $2^n$ isolated points $p_i\in T^n, \, i=1, \dots, 2^n.$ The quotient space 
$(T^n\setminus\Fix(G))/G$ can be viewed as the interior of an $n$-dimensional $C^{\infty}$ 
manifold $K(n)$ with boundary $\displaystyle \partial K(n)=\sqcup_{i=1}^{2^n} L_i,$ where 
$L_i\simeq \RR\PP^{n-1}$ is the link of the singularity of $T^n/G$ at the image of $p_i.$

Let 
$$
V=\bigoplus_{j=1}^{n-1}\left(H^{j}(T^n)\otimes H^0(K(n))\right)
\subseteq \bigoplus_{j=1}^{n-1}H^{j} (\HH_0).
$$
Note that $\inc_0^* $ restricted to $V$ coincides with the diagonal embedding 
of $V$ into a direct sum of $2^n$ isomorphic copies
$$
\bigoplus_{i=1}^{2^n}\left( \bigoplus_{j=1}^{n-1} \left(H^{j}(T^n)\otimes 
H^0(L_i)\right)\right)\subseteq \bigoplus_{j=1}^{n-1}H^{j} (E(\RR)).
$$ 
In particular, $\inc_0^* $ restricted to $V$ is injective.

\begin{lemma}
\label{better-estimate}
Let $X$ be a Smith--Thom maximal abelian variety of dimension $n$. Then 
\begin{equation}
\label{improved-rank-mu}
\sum_{k=0}^{n-1}\rank(\mu_k)\geq\sum_{k=0}^{n-1}\rank(\inc^k)+\Dim \inc_0^*(V).
\end{equation}
\end{lemma}
\proof
It suffices to show that $\im\inc^*\cap \inc_0^*(V)=\{0\}.$ To see this, fix $i$ with 
$1\leq i\leq 2^n$ and observe that 
\begin{equation}
\label{improve-tori}
\im\inc_i^*\cap \left( \bigoplus_{j=1}^{n-1} \left(H^{j}(T^n)\otimes 
H^0(L_i)\right)\right)=\{0\}.
\end{equation}

Indeed, by the third item of Theorem \ref{totaro-basis} and Lemmas 3.1, 3.2 and 
3.3 in \cite{square}, every nonzero element of $\im(\inc^*_i)$ can be written as 
$$
w=\sum_{j,\ell}b^j\left(b^{\ell}v_\ell+b^{\ell-1}\Sq^1v_\ell+\cdots\right),
$$
where $v_\ell=u_\ell \otimes 1 \in H^\ell(T^n)\otimes H^0(L_i).$ 
Consequently, since for a torus the Cartan formula yields
$\Sq^m v_{\ell}=0$ for each $m,\ell\geq 1$, we conclude that every
$w\in\im\inc_i^{\geq 1}$ is divisible by $b$ and therefore belongs to
$\displaystyle H^*(T^n)\otimes H^{>0}(L_i)$, which proves the claim.

Since all groups are finite-dimensional vector spaces over $\FF_2$, the maps
$\mu_k$ and $\mu^k$ are dual and therefore have the same rank. This completes
the proof of the lemma.
\qed


\subsubsection{Proof of Theorem \ref{tori}}


By \cite[Theorem 1.1]{hypersurfaces}, the first assertion is immediate when $X$
is not Smith--Thom maximal. We may therefore assume that $X$ is maximal, and let
$n=\Dim_\CC X\geq 2$.

\smallskip

Let 
$$
E_n:=
\begin{cases}
\displaystyle
\sum_{k=0}^{\frac{n-1}{2}}\sum_{j=0}^{2k} \beta_j +
\sum_{k=1}^{\frac{n-1}{2}} \sum_{i=0}^{k-1}\beta_i(X(\RR))+\frac{\beta_{\rm odd}}{2},\, \text{if $n$ is odd}\\
\displaystyle
\sum_{k=1}^{\frac{n}{2}}\sum_{j=0}^{2k-1} \beta_j +\sum_{k=1}^{\frac{n}{2}} \sum_{i=0}^{k-1}\beta_i(X(\RR))
,\, \text{if $n$ is even}.
\end{cases}
$$

From \eqref{deficiency-odd_mu}, \eqref{deficiency-even_mu}, and 
Lemma \ref{better-estimate}, we have 
\begin{equation}
\label{refined-estimate}
\defi(X^{[2]})=4\left( \sum_{k=0}^{n-1}\rank(\mu_{k})-E_n\right)
\geq\, 4\left(\sum_{k=0}^{n-1}\rank(\inc^k)+\Dim \inc_0^*(V)-E_n\right).
\end{equation}
Notice now that
$$
\sum_{k=0}^{n-1}\rank(\inc^k)-E_n=\frac14D_n.
$$
Since the restriction of $ \inc_0^*$ to $V$ is injective, we find 
$$
\Dim \inc_0^*(V)=\Dim V= 2^n-2.
$$
From \eqref{refined-estimate} we find
\begin{equation}
\label{new-estimate}
\defi(X^{[2]})\geq D_n+4(2^n-2).
\end{equation}
Thus, the new estimate \eqref{new-estimate} gives the improved bounds 
$\defi(X^{[2]})\geq 4$ for $n=2,$ and  $\defi(X^{[2]})\geq 32$ for $n=4.$ 
Since $D_n>0$ for $n\neq2, 4,$ it follows that $\defi(X^{[2]})>0$ for every $n\geq 2.$ 
This proves the first assertion.

\smallskip

For the second assertion, let $(X_n,\Conj_n)$ be any sequence of maximal real
abelian varieties with $\Dim_\CC X_n=n$. Lemma
\ref{bound+monotonicity} implies that there exists $\epsilon>0$
\footnote{One can take $\epsilon=6$.} such that
$
\displaystyle
\defi(X_n^{[2]})\geq 2^{2n-\epsilon},
$
and so 
$$
\ln (\defi(X_n^{[2]}))\geq (2n-\epsilon)\ln 2.
$$
Therefore, we have $\displaystyle
\liminf_{n\to\infty} \frac{\ln (\defi(X_n^{[2]}))}{n}\geq 2\ln 2.
$
\qed


\subsection{Products of curves and surfaces}


As an example of the effectiveness of Theorem \ref{defect-mv}, we prove here 
Theorem \ref{product-not-max}.

\proof [Proof of Theorem \ref{product-not-max}] 

By \cite[Theorem 1.1]{hypersurfaces}, it remains only to consider the case in 
which both products are Smith--Thom maximal. This is 
equivalent to requiring  $X$, $Y$, and $C$ to be Smith--Thom maximal. 
Under these assumptions, since $\beta_1(X)=\beta_1(Y)=0$, we find  
\begin{align}
\label{beta2}
2\beta_0(X(\RR))+\beta_1(X(\RR))=&\,2+\beta_2(X),\notag\\
2\beta_0(Y(\RR))+\beta_1(Y(\RR))=&\,2+\beta_2(Y).
\end{align}
In particular, $Y(\RR)\neq \emptyset.$ Also, if $g$ is the genus of $C,$  
maximality gives $\beta_0(C(\RR))=\beta_1(C(\RR))=g+1.$ 

\smallskip

To prove the first item, we first apply the K\"unneth formula to find the relevant Betti 
numbers of $X\times C$ and its real locus:
$$
\begin{aligned}
\begin{cases}
\beta_0(X\times C)=1,\\
\beta_1(X\times C)=2g,\\
\beta_2(X\times C)=1+\beta_2(X),\\
\beta_3(X\times C)=2g\beta_2(X),
\end{cases}
\end{aligned}
$$
and
$$
\begin{aligned}
\begin{cases}
\beta_0((X\times C)(\RR))=\beta_0(X(\RR))(g+1),\\
\beta_1((X\times C)(\RR))=(\beta_0(X(\RR))+\beta_1(X(\RR)))(g+1).
\end{cases}
\end{aligned}
$$

The remaining Betti numbers are determined from the above lists by Poincar\'e duality. 
Applying Theorem \ref{defect-mv} together with \eqref{beta2} yields
\begin{align*}
\frac14 \defi\big((X\times C)^{[2]}\big)\geq  
&\, \sum_{k=0}^{1} \sum_{i=0}^{k}\beta_i((X\times C)(\RR))-
\sum_{l=0}^{1}\sum_{j=0}^{2l} \beta_j (X\times C)\\
&\, -\frac12\beta_{\rm odd}(X\times C)\\
= &\, (3\beta_0(X(\RR))+\beta_1(X(\RR)))(g+1)\\ 
& \, - (3+2g+\beta_2(X)) - (2g+g\beta_2(X))\\
=&\, \beta_0(X(\RR))(g+1) - 1-2g.
\end{align*}
Since $\beta_0(X(\RR))\geq2,$ it follows that $\defi((X\times C)^{[2]})>0.$ 

\smallskip

To prove the second item, we proceed in a similar manner. 
The K\"unneth formula gives the following 
Betti numbers of $X\times Y$ and its real locus:
$$
\begin{cases}
\beta_0(X\times Y)=1,\\
\beta_1(X\times Y)=0,\\
\beta_2(X\times Y)=\beta_2(X)+\beta_2(Y),\\
\beta_3(X\times Y)=0,\\
\beta_4(X\times Y)=2+\beta_2(X)\beta_2(Y),\\
\end{cases}
$$
and 
$$
\begin{cases}
\beta_0((X\times Y)(\RR))=\beta_0(X(\RR))\beta_0(Y(\RR)),\\
\beta_1((X\times Y)(\RR))=\beta_0(X(\RR))\beta_1(Y(\RR))
+\beta_1(X(\RR))\beta_0(Y(\RR)),\\
\beta_2((X\times Y)(\RR))=2\beta_0(X(\RR))\beta_0(Y(\RR))
+\beta_1(X(\RR))\beta_1(Y(\RR)),
\end{cases}
$$
while the remaining Betti numbers are determined by Poincar\'e duality.
Applying again Theorem \ref{defect-mv} together with \eqref{beta2} yields 
\begin{align*}
\frac14 \defi\big((X\times Y)^{[2]}\big)\geq  
&\, \sum_{k=1}^{2}\sum_{i=0}^{k-1}\beta_i((X\times Y)(\RR))
-\sum_{l=1}^{2}\sum_{j=0}^{2l-1} \beta_j(X\times Y)\\
= &\, (2\beta_0(X(\RR))\beta_0(Y(\RR))+\beta_0(X(\RR))\beta_1(Y(\RR))\\ 
& \, +\beta_1(X(\RR))\beta_0(Y(\RR))) - (2+\beta_2(X)+\beta_2(Y))\\
=&\, 2(\beta_0(X(\RR))-1)(\beta_0(Y(\RR))-1) \\
&\, +\beta_1(X(\RR))(\beta_0(Y(\RR))-1)+\beta_1(Y(\RR))(\beta_0(X(\RR))-1).
\end{align*}

By assumption, we have $\beta_0(X(\RR))\geq 2$. The above bound 
implies that $\defi\bigl((X\times Y)^{[2]}\bigr)>0$ if $\beta_0(Y(\RR))\geq 2$. 
If $\beta_0(Y(\RR))=1,$ then the maximality condition~\eqref{beta2} gives 
$\beta_1(Y(\RR))=\beta_2(Y).$

Since $Y$ is projective, its hyperplane class is nonzero in $H^2(Y,\QQ).$ 
Thus, $b_2(Y)>0$, and therefore $\beta_2(Y)>0$. 
Hence the lower bound is again strictly positive.
\qed


\bibliographystyle{alpha}

\end{document}